%% file: 1033-public.tex
\documentclass[11pt,twoside,leqno]{amsart}

\usepackage{amssymb,enumerate}
\usepackage{epsfig}
\usepackage{cherlinshelah}
\begin{document}
\input{PS-top}

\maketitle
\thispagestyle{empty}
\input{PathSolidity00} 
\input{PathSolidity01} 
\input{PathSolidity02} 
\input{PathSolidity03} 
\bibliographystyle{plain}
\goodbreak
\input{forbidden.bbl}

\medskip
\end{document}

%% file: PS-top.tex

\title[Block Path Solidity]{Universal Graphs with a Forbidden Subgraph: Block Path Solidity}
\author[G.~Cherlin and S.~Shelah]{
Gregory Cherlin
\\Department of Mathematics, \\Rutgers University, U.S.A.\\
\and\\
Saharon Shelah\\
Department of Mathematics\\ Hebrew University, Jerusalem, Israel\\
and Rutgers University, Piscataway, NJ 08854
}
\begin{abstract}
Let $C$ be a finite connected graph for which there is a countable universal $C$-free graph,
and whose tree of blocks is a path.
Then the blocks of $C$ are complete. This generalizes a result of F\"uredi and Komj\'ath, and
fits naturally into a set of conjectures regarding the existence of countable $C$-free graphs,
with $C$ an arbitrary finite connected graph.
\end{abstract}
\thanks{First author supported by NSF Grant DMS-1101597} 
\thanks{The second author would like to thank the NSF, grant no. DMS 1101597, for support of this paper. Publication 1033.}

%% file: PathSolidity00.tex

\section*{Introduction}

\subsection*{\quad The main theorem}
The general problem to be considered here is the following.

\begin{problem}[Universality with 1 Forbidden Subgraph]
Let $C$ be a finite connected graph. When is there a universal $C$-free graph?
\end{problem}

We anticipate that the problem as posed has an explicit solution, but not a very simple one.
The present paper is devoted toward the proof of one concrete result which is part of a 
general plan of attack on the problem. That result reads as follows.

\begin{theorem}[Main Theorem]\label{MainTheorem}
Let $C$ be a finite connected block path, and suppose there is a (weakly) universal $C$-free graph.
Then the blocks of $C$ are complete.
\end{theorem}

The terminology will need to be explained in detail, but we first lay out the context for the result.
The general plan of attack referred to consists mainly of the following two conjectures.

\begin{conjecture}
Let $C$ be a finite connected graph allowing a countable universal $C$-free graph.
Then 
\begin{itemize}
\item {(\it Solidity Conjecture) }The blocks of $C$ are complete;
\item ({\it Pathlike Conjecture}) $C$ may be obtained from a block path by adjoining paths to some of the vertices, with at most one path joined to each vertex.
\end{itemize}
\end{conjecture}

Our main theorem may be phrased as follows: the second conjecture implies the first.
Given both conjectures, what then remains is to analyze the graphs referred to as ``pathlike''
here, under the assumption that all blocks are complete.

Recall that the blocks of a graph are its maximal 2-connected subgraphs, and that there is associated
to any connected graph a ``tree of blocks'' whose vertices are the blocks and cut vertices of the original graph, with edges corresponding to incidence. By a {\it block path} we mean a graph whose tree of blocks forms a path. Such graphs are certainly pathlike, but we also will need to consider 
block paths with ``whiskers,'' as described in the second conjecture.

Now we review the terminology relating to universality. 

\begin{definition*}[Universality]
With $C$ a given {\it forbidden}
subgraph, a graph $\Gamma$ is  {\it $C$-free} if $\Gamma$ contains no subgraph isomorphic to $C$.

If $\Gamma$ is countable and $C$-free, then $\Gamma$ is
\begin{itemize}
\item {\it strongly universal} 
if every countable $C$-free graph is isomorphic to an induced subgraph
of $\Gamma$;
\item {\it weakly universal}
if every countable $C$-free graph is isomorphic to a  subgraph
of $\Gamma$
(such an isomorphism will be called an embedding, or more explicitly, 
an embedding as a subgraph).
\end{itemize}
\end{definition*}

The most natural notion of universality is the strong version. 
In proving the existence of universal
graphs we always aim at strong universality. But when proving nonexistence it is 
more satisfactory to prove the nonexistence of a {\it weakly} universal graph, and as the marginal
cost of this refinement tends to be low, this is what we usually aim at. 
In particular, we have stated the main theorem
in this sharper form. We also stress that we work 
with the class of countable $C$-free graphs throughout. 
Other cases are of interest but involve different issues and a broad range of techniques
(cf.~\cite{KP91,Dz05}).

The proof of the main theorem is not very complicated. We use two techniques: the {\it hypergraph
method} of F\"uredi and Komj\'ath \cite{FK97a}, and two {\it pruning techniques}, one of which has made
an appearance in \cite{CS05} in a more specialized form.
The pruning techniques allow a certain kind of inductive argument to be carried out, typically reducing
a riot of uncontrollable cases to a menagerie of manageable cases. We expect to make substantial further use of the pruning method again in more complicated situations, so one of our goals here is
to set this up for future reference.

We view the explicit classification of all finite connected constraint graphs $C$ allowing a universal
$C$-free graph as an ambitious classification project.  But what really interests us
is the following more qualitative question.

\begin{problem}[Universality with finitely many Forbidden Subgraphs]
Is there an algorithm which will decide, for any finite set $\CC$ of finitely many connected graphs,
whether there is a universal $\CC$-free graph?
\end{problem}

This problem can be approached from many sides. 
It can be shown that if one forbids not just subgraphs,
but induced subgraphs, then one arrives at a still more general question, but one for which
there is a direct proof of {\it algorithmic undecidability} by encoding Wang's domino problem
\cite{Ch11}. 
On the other hand if one restricts attention to graphs of bounded degree---that is, one allows a finite set of forbidden connected graphs, one of which is a star---then
the problem becomes decidable, though one would not expect a completely explicit classification
at that level of generality.

\subsection*{\quad Some prior work}

For the purposes of the present article, the outstanding prior result
is due to F\"uredi and Komj\'ath.

\begin{fact*}[2-Connected Constraints, \cite{FK97a}]\label{FurediKomjath}
Let $C$ be a finite 2-connected graph for which there is a weakly universal $C$-free graph.
Then $C$ is complete.
\end{fact*}

Conversely, it is very well known that when $C$ is complete, there is indeed a universal $C$-free graph
\cite{He71}.

We will present the  F\"uredi--Komj\'ath {\it hypergraph technique} in detail, 
and make further use of it, in \S1. This technique uses a certain hypergraph with good properties
as a template for a construction.

F\"uredi and Komj\'ath stated a more general result.
Call a block of $C$ {\it distinguished} if it embeds as a subgraph in no other block (in particular it
may not be isomorphic to another block). 
They show that the existence of a $C$-free universal graph forces the
distinguished blocks of $C$ to be complete. As it turns out, we need a different variation on their
theme.

A result whose relevance to our current concerns is less obvious is the following.

\begin{fact*}[Tree constraints, \cite{CT07,CS05}]
Let $C$ be a finite tree for which there is a universal $C$-free graph.
Then $C$ is either a path, or may be obtained from a path by adjoining one further edge.
\end{fact*}

The point here is the method of proof used, rather than the result obtained. 
The analysis made use of a form 
of what we will refer to as {\it corner pruning} below. We also introduce a second pruning technique
we call {\it symmetric local pruning},  which has not previously seen the light of day.

The pruning method supports an inductive analysis, 
which in the case of one forbidden tree led us to consider 
14 possible cases as the base of the induction.  
One may wonder whether one can deal
with an arbitrary constraint graph by such a method.  We believe this is feasible.

A larger view of the wide range of problems connected with universal graphs  may be found in the surveys
\cite{KP91,Dz05}. 
\relax From the point of view of those surveys, the restriction to {\it countable} universal graphs determined by a {\it finite} set of constraints covers many cases of interest, but by no means all.
The article \cite{Ch11} is in large part a survey of the universality problem 
restricted to the countable setting, allowing only a finite number of constraints.
By restricting to this narrower context we reach a clear algorithmic problem which seems to us
to raise a fundamental issue about the subject: namely, is it an art or a science?

%% file: PathSolidity01.tex

\section{The F\"uredi--Komj\'ath  Hypergraph Construction}

The proof of the main theorem will be given in three stages, where at each stage we first introduce
a general technique, then apply that technique to analyze a hypothetical minimal counterexample to 
our main theorem.

In the present section we will discuss the F\"uredi--Komj\'ath hypergraph construction and 
apply it to obtain  the following special case of our main theorem. Since F\"uredi and Komj\'ath alread
used their method to treat the case of a single block, this amounts to finishing the base case of
an inductive analysis to be taken up further in succeeding sections.

\begin{lemma}\label{2Blocks}
Let $C$ be a finite connected graph with exactly two blocks $(B_1,B_2)$,
and with 
$$|B_1|\le |B_2|$$
Suppose that there is a weakly universal $C$-free graph. 
Then $B_1$ is a complete graph.
\end{lemma}

We first present the hypergraph construction of F\"uredi and Komj\'ath 
\cite{FK97a}.

\goodbreak
\subsection{The Hypergraph construction}

\begin{definition}
\indent\par
A {\em cycle} in a hypergraph is a sequence of distinct
vertices and edges $(v_0,E_0,v_1,\dots,v_{n-1},E_{n-1})$ with
$v_i,v_{i+1}\in E_i$---here we take $(i+1)$ mod $n$, and $n\ge 2$.

The {\em girth} of a hypergraph is the length $n$ of the shortest cycle
(or $\infty$). 

A hypergraph is {\em $k$-uniform}
if its hyperedges consist of $k$ points.
\end{definition}

Note that if the girth of a hypergraph is greater than $2$ then distinct
hyperedges meet in at most one vertex. The following is a slight
variation on a result of \cite{FK97a}. 

\begin{lemma}\label{fk1:hypergraph}
For any $k,g$ with $k\ge 2$ there is some $N=N(k,g)$ and 
a $k$-uniform hypergraph of
girth at least $g$ on the vertex set $\Nn$, with hyperedges $E_i$ ($i\in
\Nn$)
satisfying
$$ N+(i-1), N+i\in E_i\includedin [0,\dots,N+i]$$
for all $i$.
\end{lemma}
\begin{proof}
We will impose an additional condition on the hypergraph:
$$\hbox{No vertex belongs to more than $k$ hyperedges}$$

We proceed inductively. For the moment, let $N$ be arbitrary, and
suppose we have constructed a $k$-uniform 
hypergraph on $I_i=[0,\dots,N+i-1]$
with hyperedges $E_0,\dots,E_{i-1}$ satisfying all relevant
conditions up to this point. We wish to select $E_i$.

Let $V_i$ be the set of vertices in $I_i$ belonging to exactly $k$ of
the hyperedges $E_j$ for $j<i$. 
By counting all pairs $(u,E_j)$ with
$u\in I_i$, $j<i$, and $u\in E_j$ we find 
$|V_i|\le i$. Letting $V_i'=I_i\setminus V_i$ we have $|V_i'|\ge N$.
Notice that $N+i-1\in V_i'$ since this vertex belongs only to $E_{i-1}$.

Consider the ordinary graph $G_i$ induced on $I_i$ by taking edges
$(u,v)$ whenever $u,v\in E_j$ for some $j<i$. This has vertex degree
bounded by $k(k-1)$. Hence in the graph metric we have a bound on the
order of balls of radius $g$, and for $N$ sufficiently large we may
select a subset $X\includedin V_i'$ of cardinality $k-1$, with
$N+i-1\in X$, and with $d(u,v)>g$ in $G_i$ for $u,v\in X$.
Set $E_i=X\union \{N+i\}$. Then our conditions are all preserved.
\end{proof}

\subsection{Application: Lemma \ref{2Blocks}}

\begin{proof}[Proof of Lemma \ref{2Blocks}]
We have $C=(B_1,B_2)$, where $B_1,B_2$ are blocks meeting at the unique cut
vertex $v_*$ of $C$. Let $n_i=|B_i|$. We have assumed that 
$$n_1\le n_2$$

We will suppose that the block $B_1$ is not complete, and we aim to show that
there is no weakly universal $C$-free graph. As $B_1$ is not complete we have
$$n_1\ge 4$$

Let $k=n_1+1$, let $g$ be greater than the maximum order of a block of $C$,
and let $\Gamma$ be a $k$-hypergraph of girth at least $g$ 
with the properties of Lemma \ref{fk1:hypergraph}.
We will label the vertices of $\Gamma$ as  $(u_i:i\in \Nn)$ (so in fact $u_i=i$).

Divide each hyperedge $E$ of $\Gamma$ into
$$\mbox{$E=E^1\djunion E^2$ with $|E^1|=n_1-1$, $|E^2|=2$, and $(\max E)\in E^1$}$$

Let $G_0$ denote the graph on $\Gamma$ in which the induced graphs on each set $E^1$ are cliques
of order $n_1-1$, and there are no additional edges. Let $G_1$ be the graph obtained from $G_0$ by
attaching one clique $K_v$ of order $n_2+1$ freely to each  vertex $v$ of $G_0$.

It is clear that $G_1$ is $C$-free; we will need a sharper statement given as Claim \ref{Gepsilon} below.

\begin{claim}
Let $G$ be a $C$-free graph, let $n\in \Nn$, and let $f_1,f_2:G_1\into G$ be embeddings of $G_1$ into $G$ which agree on $u_i$ for $i<n$. Suppose there is a hyperedge $E$ with $u_n=\max E$.
Then $f_1,f_2$ agree on $u_n$. 
\end{claim}

We let $u=f_1(u_n)$ and $u'=f_2(u_n)$, which we suppose distinct. 
Let $B'=f_1[E^1\setminus \{u_n\}]=f_2[E^1\setminus \{u_n\}]$
and $B=B'\union \{u,u'\}$. 
Then the induced graph on $B$ in $G$ contains all edges except possibly
$(u,u')$. 
As $B_1$ is not complete, there is an embedding $j:B_1\to B$ as a subgraph. 
Let $v=j(v_*)$. 
Then $v=f_i(u)$ for some $u\in E^1$ and some $i=1$ or $2$ (or both).
Thus $K=f_i[K_u]$ is a clique of order $n_2+1$ containing $v$ and meeting $B$ in at most one other vertex. 
Let $K'\includedin K$ be a clique of order $n_2$ meeting $B$ in $\{v\}$ alone. Then $(B,K')$ contains a copy of $C$,
so $C$ embeds into $G$, and we have a contradiction. This proves our claim.

\medskip
Now we extend the graph $G_1$ in a variety of ways. 

Fix an edge $e$ of $B_1$ containing the cut vertex $v_*$ of $C$, and let $B_1'=B_1\minusedge e$.
Let  $\EE$ be the set of hyperedges of $\Gamma$.  
For $\varepsilon:\EE\to \{0,1\}$ arbitrary, 
we will extend $G_1$ to a graph $G_\varepsilon$ as follows.
\begin{itemize}
\item []For each hyperedge $E$ of $\Gamma$
\begin{enumerate}
\item If $\varepsilon(E)=1$ then attach an edge to $E^2$;
\item If $\varepsilon(E)=0$ then attach a copy of $B_1'$ to $E^2$ 
with the ends of the deleted edge $e$ corresponding to the vertices of $E^2$.
\end{enumerate}
\end{itemize}

\smallskip
\begin{claim}\label{Gepsilon}
Each of the graphs $G_\varepsilon$ is $C$-free.
\end{claim}

For $E$ a hyperedge of $\Gamma$, let $\hat E$ be the portion of $G_\varepsilon$ supported by $E$,
namely the induced graph on the union of $E$ together 
with all the attached cliques $K_v$ ($v\in E$), and
also (when $\varepsilon(E)=0$) the attached copy of $B_1'$. 
By the choice of the girth $g$, if $f:C\to G_\varepsilon$, 
then the image of either block $f[B_i]$ must lie in one of the sets $\hat E$ (not necessarily unique, 
since the attached cliques
$K_v$ are shared by several of the $\hat E$). 

As neither block $B_1,B_2$ can be mapped into a copy of $B_1'$, or into a clique of order $n_1-1$, or 
into a single edge (on $E^2$),
these blocks must both go into an attached clique $K_v$. 
But then they must go into the same clique $K_v$,
and as $n_1>2$ this is impossible.

Now the concluding argument follows a well worn path: namely, we have constructed
uncountably many suitably incompatible $C$-free graphs, and therefore no countable
$C$-free graph can be weakly universal. We give this final argument in detail.

Suppose there is a countable weakly universal $C$-free graph $G$, 
and choose embeddings
$f_\varepsilon:G_\varepsilon\to G$ of each $G_\varepsilon$ as a subgraph of $G$. 
Let $N$ be the parameter associated with the hypergraph $\Gamma$, with the property:
$$\mbox{For any $n\ge N$, there is a hyperedge $E$ of $\Gamma$ with $\max E=u_n$}
$$

As $C$ is countable,  there will be a pair of distinct $\varepsilon_1$, $\varepsilon_2$ for which the corresponding
embeddings agree on $u_i$ for $i<N$, and hence agree for all $i$ in view of Claim 1. 

Now consider a hyperedge $E$ with $\varepsilon_1(E)\ne\varepsilon_2(E)$.
We may suppose 
$$
\varepsilon_1(E)=0;\ \varepsilon_2(E)=1
$$
Then the image of $E^2$ in $G$ contains an edge. 
Now consider the copy $B_1^*$ of $B_1'$ attached to $E^2$ in $G_{\varepsilon_1}$, 
and the vertex $v^*$ of $B_1^*$ corresponding to the cut vertex of $C$. Write $f$ for $f_{\varepsilon_1}$.
Then $f[B_1^*\union K_{v^*}]$ contains a copy of $C\minusedge e$ where $e$, 
the deleted edge of $B_1$, corresponds to the edge on $f[E^2]$ in $G$. 
That is, we have now embedded
$C$ into $G$, arriving at a contradiction.
\end{proof}

%% file: PathSolidity02.tex

\section{Corner Pruning}

We now aim at the following reduction of our main theorem.

The {\it length} of a block path will be defined as the number of blocks.

\begin{lemma}\label{CornerPruning}
Suppose that $C=(B_1,\dots,B_\ell)$ is a block path of length $\ell$
allowing a weakly universal graph, and having a block which is not complete.
Suppose further that the  length $\ell$ is minimal, that $|B_1|\le |B_\ell|$, and that
if $B_\ell$ is isomorphic to a subgraph of $B_1$, then $B_\ell$ is isomorphic to $B_1$.
For $i<\ell$, let $v_i$ be the cut vertex between $B_i$ and $B_{i+1}$.
Then the following hold.
\begin{enumerate}
\item $B_1$ is not complete
\item $\ell\ge 3$
\item $B_i$ is complete for $1<i<\ell$
\item If $B_\ell$ does not embed in $B_1$, then $B_\ell$ is complete
\item The induced subgraph $(B_1,B_2\setminus \{v_2\})$ of $C$ embeds into the induced subgraph
$(B_2\setminus \{v_1\},B_3,\dots,B_\ell)$.
\end{enumerate}
\end{lemma}

We view the last of these conditions is a weak form of symmetry.  One case to keep in mind is that in which
the length is $3$ and $B_1,B_3$ are isomorphic.

\goodbreak
\subsection{The method of corner pruning}

\begin{definition}
\indent
\par
1. A {\em segment} of a graph $C$ is a connected subgraph which is  a union of blocks.

2. A {\em corner $C_v$} of a graph $C$ is a segment of the form $\{v\}\union C'$ where $v$ is a cut vertex and $C'$ is one of the  connected components of $C\setminus \{v\}$.
Note that $C_v$ contains a unique block $B$ of $C$ with $v\in B$, and that the pair $(v,B)$ 
determines the corner.
We call $v$ the {\sl root} of $C_v$, and $B$ its {\sl root block}.
Note that a corner will frequently be treated as a graph with base  point $v$ (or briefly: a {\em pointed graph}). For pointed graphs we use the notation
$$(v,C)$$
In particular, we may consider embeddings of one corner into another either as a subgraph, or as a pointed subgraph. 
\end{definition}

\begin{definition}[Pruning]
Let $\Sigma$ be a set of pointed graphs, $C$ a graph, and $\CC$ a finite set of graphs.

1. A corner $C_v$  of $C$ is {\em pruned by} a pointed graph $(u,S)$  if there is an embedding of $(v,C)$ into 
$(u,S)$ as a pointed subgraph.

2. The $\Sigma$-pruned graph $C_\Sigma$ is the graph obtained from $C$ by deleting the set of vertices in
$$\Union_{(u,S)\in \Sigma}\{C_v\setminus \{v\}\suchthat \mbox{$(v,C_v)$ is pruned by $(u,S)$}\}$$

Thus we do not delete the base point of a pruned corner, only the remainder.

3. $\CC_\Sigma=\{C_\Sigma\suchthat C\in \CC\}$

4. Generally we write $C'$ and $\CC'$ for the pruned graph or set of graphs, after specifying the set 
$\Sigma$.
\end{definition}

We focus here on single constraints $C$ and we prune by a single {\it minimal} corner.
But there is a distinction even in simple cases between pruning by a set of corners taken together,
and pruning by a sequence of corners individually and consecutively. We will not require 
the notion in full generality
for our present purposes, but this is likely to come into play in more elaborate analyses.

In \cite{CS05}, we dealt with  the case of one forbidden tree, and we pruned only leaves. 
The simplest case of a corner would be a block occurring as a leaf in the tree of blocks.
The proof of the next result is much the same as in \cite{CS05}.

\begin{lemma}[Pruning Induction]
Let $\CC$ be a set of graphs, and $\Sigma$ a set of pointed graphs. If there is a countable 
universal $\CC$-free graph (in either the weak or strong sense) then there is a countable universal
$\CC_\Sigma$-free graph.
\end{lemma}

\begin{proof}
If $G$ is any graph, define $G^+(\Sigma)$ as the graph obtained by freely adjoining infinitely many copies of each pointed graph $(u,S)$ in $\Sigma$ to each vertex $v$ of $G$, identifying $u$ and $v$.
(In particular for $v$ a single vertex, viewing $v$ as a trivial graph, we have the notation $v^+(\Sigma)$.)

If $G$ is $\CC_\Sigma$-free, then $G^+(\Sigma)$ is $\CC$-free: if $C\in \CC$ embeds as a subgraph into 
$G^+(\Sigma)$, then the part of $C$ lying in $G$ would contain $C_\Sigma$.

So now suppose there is a weakly or strongly universal $\CC$-free graph $\Gamma$,
and let $\Gamma_\Sigma$ be the induced graph on the set
$$\{v\in \Gamma\suchthat \mbox{$v^+(\Sigma)$ embeds into $\Gamma$ over $v$\}}$$
We will check that $\Gamma_\Sigma$ is $\CC_\Sigma$-free universal, in the same sense.

Certainly $\Gamma_\Sigma$ is $\CC_\Sigma$-free, as otherwise we could reattach the pruned corners in
$\Gamma$.

So let $G$ be $\CC_\Sigma$-free, and embed $G^+(\Sigma)$ into $\Gamma$, either as a subgraph or as an induced subgraph, as the case may be. Then $G$ goes into $\Gamma_\Sigma$.
\end{proof}

\begin{remark}
There is also some use for a more sensitive notion of pruning, in which we prune only segments which embed into the given pointed graphs as {\em induced subgraphs.} 
But this would be relevant only in
proving the nonexistence of strongly universal graphs, while we aim at proving nonexistence
for weakly universal graphs.
\end{remark}

The classification of forbidden trees $C$ allowing a universal $C$-free graph
comes down to the following, by leaf pruning.

\begin{fact*}[\cite{CS05}]
Let $T$ be a tree which becomes either a path or a near path on removal of its leaves.
If there is a weakly universal $T$-free graph, then $T$ is a path or a near path.
\end{fact*}

This amounts to the base of an induction, and occupies the bulk of \cite{CS05};  the rest of 
the induction is purely formal, as we have seen.

We will now proceed similarly with the proof of our main theorem. 
Corner pruning plus the hypergraph
construction will not do everything, but will leave a definite configuration suitable for further analysis
by a third method.

It should be noted that  the hypergraph construction could do a good deal more than we have done 
with it---but only via partial overlap with cases handled more thoroughly by pruning.

\subsection{Application: Lemma \ref{CornerPruning}}
\begin{proof}[Proof of Lemma \ref{CornerPruning}]

We are supposing that $C$ is  
a block path with blocks $(B_1,\dots,B_\ell)$, 
allowing a countable weakly universal $C$-free graph, and having some incomplete
block, with the length $\ell$  minimal. 

We  may suppose further that
$$|B_1|\le |B_\ell|$$
and that if $B_\ell$ embeds into $B_1$, then the two blocks are isomorphic.

By the result of F\"uredi and Komj\'ath, $\ell\ge 2$.

The block $B_1$ is a corner of $C$ and we may prune it. Our assumptions on $B_1,B_\ell$
imply that this pruning will remove only $B_1$ and possibly $B_\ell$, 
the latter only if $B_1$ and $B_\ell$
are isomorphic. What is left after pruning is a shorter block path with similar properties, so by the
minimality of $\ell$ all of the remaining blocks are complete: that is, $B_i$ is complete 
for $1<i<\ell$, and also $B_\ell$ is complete if $B_1$ and $B_\ell$ are not isomorphic.

Since $C$ has some incomplete block, it follows that $B_1$ is incomplete.
So at this point we have
\begin{itemize}
\item[]$B_1$ is incomplete
\item []$\ell\ge 3$ (Lemma \ref{2Blocks})
\item[] $B_i$ is complete for $1<i<\ell$; and for $i=\ell$ unless $B_1\iso B_\ell$
\end{itemize}

So points $(1-4)$ of Lemma \ref{CornerPruning} have been verified.

Our final claim $(5)$ is that we have an embedding of $(B_1,B_2\setminus \{v_2\})$ into
$(B_2\setminus \{v_1\},B_3,\dots,B_\ell)$, where $v_i$ denotes the cut vertex between
$B_i$ and $B_{i+1}$.

It will be useful to bear in mind that the corners of $C$ are its terminal
segments 
\begin{align*}
R_j=(B_j,B_{j+1},\dots,B_\ell)&\mbox{ with base point $v_{j-1}$, and}\\
L_j=(B_j,B_{j-1},\dots,B_1)&\mbox{ with base point $v_j$,} 
\end{align*}
where we write ``$R$'' and ``$L$'' to suggest ``right'' and ``left''. 

Now we prune the corner $R_3=(B_3,\dots,B_\ell)$ with base point $v_2$.
If $B_1$ remains after pruning, then by the minimality of the length $\ell$, $B_1$ is complete, a contradiction.

If $B_1$ does not remain after pruning, then it meets, and hence lies within, some corner pruned by $R_3$, which must be of the form
$L_j=(B_j,\dots,B_1)$ with base point $v_j$
(taking the blocks in reverse order). So $L_j$ embeds into the corner $R_3$, 
with $v_j$ corresponding to $v_2$. 

Suppose  first that $j=1$.  Then we have an embedding of $B_1$ into $B_3$ with $v_1$ going to $v_2$.
As $B_2$ is complete, we may extend this to an embedding of $(B_1,B_2\setminus \{v_2\})$ into
$(B_2\setminus \{v_1\},B_3)$, taking $v_2$ to $v_1$, proving our claim.

Now suppose $j>1$. Then our embedding takes $(B_2,B_1)$ into $(B_3,\dots,B_\ell)$
with $v_2$ fixed if $j=2$, and with $v_2$ not in the image if $j>2$. So claim $(5)$ 
holds in either case.
\end{proof}

We will see in the next section that the weak symmetry condition $(5)$ allows another kind of pruning.

%% file: PathSolidity03.tex

\section{Local and symmetric Pruning}

In this section we aim to complete the proof of the main theorem by dealing
with the configuration described in Lemma \ref{CornerPruning}.
We introduce another, more subtle, pruning technique.
At this point we will confine our theoretical discussion 
to the case of a single constraint, though no doubt this 
 tool is useful in greater generality. 
 
\subsection {The method of local  pruning}

{\it Local} pruning is a way of removing a single corner. 
We first give the definition in pragmatic terms, and then look for reasonable
conditions sufficient for its application.

\begin{definition}
Let $(v,C_v)$ be a corner of the finite connected graph $C$. Let $C_v^+$ be the union of the other corners of $C$ rooted at $v$, and $C_v^-=C_v\setminus \{v\}=C\setminus C_v^+$.

1. For any graph $H$, the graph $\hat H=H*_v C_v^+$ is the suspension of $H$ with an attached
copy of $C_v^+$, constructed as follows. 
\begin{enumerate}[(a)]
\item Take the disjoint union $H\djunion C_v^+$
\item Connect $v$ to every vertex of $H$ by an edge.
\end{enumerate}

2. $(v,C_v^+)$ is said to be {\em detachable} if the following holds.
\begin{quotation}
Whenever $H$ is $C_v^-$-free, then $\hat H$ is $C$-free.
\end{quotation}
\end{definition}

\begin{lemma}Let $C$ be a finite connected graph, and $(v,C_v)$ a corner with
$(v,C_v^+)$ detachable. If  there is a countable universal $C$-free graph, in either the weak or strong sense,
then there is a countable universal $C_v^-$-free graph, in the same sense.
\end{lemma}

\begin{proof}
Let $\Gamma$ be a universal $C$-free graph, in one of the two senses. 

For each embedding $h$ of $(v,C_v^+)$ into
$\Gamma$ as a subgraph, let $\Gamma_h$ be 
$$\{u\in \Gamma\suchthat \mbox{$u\notin h[C_v^+]$ and $(u,h(v))$ is an edge}\}$$

Let $\Gamma_0=\Djunion_h \Gamma_h$ (a {\it disjoint union}) with $h$ varying over weak embeddings (as subgraphs) if we are in the weak case, or over strong embeddings (as induced subgraphs) in the strong case.
We claim that $\Gamma_0$ is countable universal $C^-_v$-free, in the corresponding sense.

As $C^-_v$ is connected and the individual $\Gamma_h$ are $C^-_v$-free, the graph $\Gamma_0$ is
$C^-_v$-free. Now we check the universality.

If $H$ is any countable $C^-$-free graph 
then we form the extension $\hat H=H*_v C_v^+$ and by hypothesis $\hat H$ is $C$-free, hence embeds into $\Gamma$. Then this embedding takes $H$ into the corresponding induced subgraph 
$\Gamma_h$ in $\Gamma_0$.
\end{proof}

\subsection{A special case: Symmetric Local Pruning}

Now we become more concrete, in the context of block paths. 
We continue to work with the notation of the previous section.

In analyzing detachability,  we must pay particular attention to ``improbable'' embeddings of a given graph
$C$ in some graph of the form $\hat H$.

\goodbreak
\begin{lemma}[Symmetric Local Pruning]\label{SymmetricLocalPruning}
Let $C$ be a block path, $B$ a block of $C$ containing two cut vertices $u,v$, 
and let $L_u,R_u,L_v,R_v$ be the corners rooted at $u$ and $v$ respectively,  
with $R_u$ and $L_v$
the ones containing the block $B$.
Suppose that $L_v\setminus \{v\}$ embeds into $R_u\setminus \{u\}$. Then 
$(v,R_v)$ is detachable.
\end{lemma}

\begin{center}
\epsfysize 1.5 true in
\epsfbox{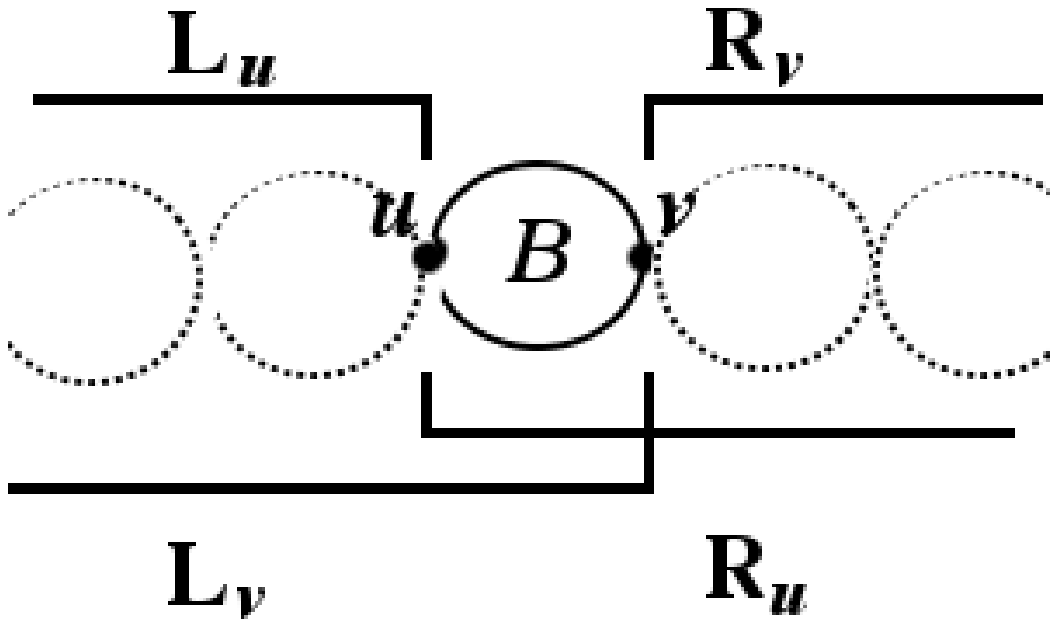}
\end{center}
\goodbreak
\begin{remark}
Here the block $B$ is common to $L_v$ and $R_u$, so one possible type of embedding would involve
a symmetry over $B$. In practice other types of embedding will also occur, so the concept of ``symmetry''
used here is very general.

In our notation, {\it L} and {\it R} stand for {\it left} and {\it right}. It is helpful to think of the tree of blocks, which is
a path, ordered so that $u<B<v$. Note that $R_v=L_v^+$.
In particular detaching $R_v$ leaves $L_v^-=L_v\setminus \{v\}$.
\end{remark}

\begin{proof}[Proof of Lemma \ref{SymmetricLocalPruning}]

We suppose that $H$ is an $L_v^-$-free
graph, and we set
$$\hat H=H*_{v}(v,R_v)$$
Our claim is that $\hat H$ is $C$-free.

We will write $(v_0,R_0)$ for the copy of $(v,R_v)$ in $\hat H$, and
fix an isomorphism
$$\iota:(v_0,R_0)\iso (v,R_v)$$

Suppose toward a contradiction that
$$f:C\iso \hat C\includedin \hat H$$
where the inclusion is as a subgraph.

If $f[R_v]\includedin R_0$ then $f[R_v]=R_0$ and $f[L_v^-]\includedin H$, a contradiction.
So
$$f[R_v]\not \includedin R_0$$
\smallskip

\goodbreak
\noindent{\it Case 1.} Suppose first that $v_0\notin f[R_v]$.

In this case we must have $f[R_v]$ contained either in $R_0\setminus
\{v_0\}$ or in $H$, and the first alternative is out of the question. Thus
$$f[R_v]\includedin H$$
In particular, $f(v)\in H$ and $f[B]\includedin H\union \{v_0\}$. 

If $f[B]\includedin H$ then $f[R_u]\includedin H$ and $L_v^-$ embeds into $H$, a contradiction.

So suppose  $v_0\in f[B]$ and $v_1=f^{-1}(v_0)$. Then as $v_1\ne v$ and $B$
is complete, we have $(R_u\setminus \{u\})\iso (R_u\setminus \{v_1\})$. 
But $f[(R_u\setminus \{v_1\})]\includedin
  H$ and thus $L_v^-$ embeds into $H$, a contradiction.
\smallskip

\noindent {\it Case 2.} Suppose that $v_0=f(v)$.

Then $f[L_v^-]$ is contained either in $H$ or in $R_0\setminus \{v_0\}$. 
As the first alternative is excluded by hypothesis, we have
$$f[L_v^-]\includedin R_0\setminus \{v_0\} $$

Since $f[R_v]\not \includedin R_0$, we have
$$\mbox{$f[R_v]\includedin H\union \{v_0\}$, with $f(v)=v_0$}$$
So $f\iota f[L_v^-]\includedin H$, a contradiction.

\smallskip
\goodbreak

\noindent {\it Case 3.} Suppose $v_0=f(v_1)$ with $v_1\in R_v\setminus \{v\}$. 

This is the most delicate case.

As $v\notin f[L_v^-]$ we again have $f[L_v^-]$ contained in  $H$ or $R_0\setminus \{v_0\}$, 
with the first alternative ruled out
by hypothesis. So we have
$$f[L_v^-]\includedin R_0\setminus \{v_0\}$$
In particular
$$f(v)\in R_0\setminus \{v_0\}$$

In what follows we are mainly concerned about the relation of $\iota f(v)$ to $v_1$.

Let $S$ be the smallest segment 
of $R_v$ containing $v$ and $\iota f(v)$. 

Suppose first that
\begin{align}
f[S]&=\iota^{-1}[S]
\end{align}

Then  $f[S]\includedin R_0$, and $v_0\in f[S]$, so since $f[R_v]\not\includedin R_0$
we find $f[R_v\setminus S]\includedin H$.

Now $f[L_v]\intersect \iota^{-1}[S]=f[L_v\intersect S]=\{f(v)\}$,
so $\iota f[L_v]\intersect S=\{\iota f(v)\}$.
Thus $\iota f[L_v^-]\includedin R_v\setminus S$,
and so $f \iota f[L_v^-]\includedin f[R_v\setminus S]\includedin H$, and we have an embedding
of $L_v^-$ into $H$, for a contradiction.

There remains the alternative
\begin{align}
f[S]&\ne \iota^{-1}[S]
\end{align}
and hence
$$f[S]\not\includedin \iota^{-1}[S]$$

Next we claim
$$f(\iota f(v))\in H$$

Otherwise, we have $f[S]\includedin R_0$. 
But $\iota f(v)\in S\intersect \iota f[S]$,  
$\iota f[S]\not\includedin S$, so $\iota f[S]$  contains the cut vertex between $S$ 
and $R_v\setminus S$ and the adjacent block of $R_v\setminus S$.
This then forces $f[R_v]\includedin R_0$, a contradiction.
So $\iota f(v)\in H$.

In particular, $v_0$ is a cut vertex of $f[S]$, and thus
$$\mbox{$v_1$ is a cut vertex of $S$.}$$
Let $S_1$ be the segment from $v$ to $v_1$ in $R_v$.

Now $f[R_v\setminus S]\includedin H$, and $f(v_1)=v$, so 
$$f[R_v\setminus S_1]\includedin H$$

Now $\iota^{-1}(v_1)$ is a cut vertex of $R_0$ lying between $v_0$ and $f(v)$.
Hence $\iota^{-1}(v_1)\in f[S]$, and $v_1\in \iota f[S]$. As $v_1\ne v$, we have
$$v_1\notin \iota f[L_v^-]$$

But $\iota f(v)\in \iota f[L_v]$, so 
$\iota f[L_v^-]\includedin R_v\setminus S_1$.
So $f\iota f[L_v^-]\includedin H$ and again we have a contradiction.
\end{proof}

\subsection{Application: The Main Theorem}

\begin{proof}[Proof of Theorem \ref{MainTheorem}]
We suppose toward a contradiction that 
$$C=(B_1,\dots,B_\ell)$$
is a block path with at least
one incomplete block, allowing a weakly universal $C$-free graph, and with the length
$\ell$ minimized.

We may suppose
$$|B_1|\le |B_\ell|$$

We claim that we may also suppose that one of the following two conditions applies
\begin{itemize}
\item $B_1\iso B_\ell$
\item $B_\ell$ is not isomorphic to a subgraph of $B_1$
\end{itemize}

If $|B_1|<|B_\ell|$ this is clear, while if  $|B_1|=|B_\ell|$, 
we are free to switch the roles of $B_1$ and $B_\ell$. So unless $B_1$ and $B_\ell$
are isomorphic, we may suppose that $B_\ell$ does not embed isomorphically into $B_1$.

So we arrive at the conditions of Lemma \ref{CornerPruning}, 
and in particular at the conclusions
that $\ell\ge 3$, $B_1$ is not complete, and that  $L_2^-=(B_1,B_2\setminus \{v_2\})$ 
embeds into
$(B_2\setminus \{v_1\},\dots,B_\ell)$. 
Taking $B=B_2$ in Lemma \ref{SymmetricLocalPruning}, we find that $R_3=(B_3,\dots,B_\ell)$ is detachable. Thus 
there is a weakly universal $L_2^-$-free graph. 
So by the case $\ell=2$ (or $\ell=1$ if $|B_2|=2$), the block $B_1$ is complete.
This contradiction completes the proof.
\end{proof}

%% file: forbidden.bbl

%% file: 1033-public.bbl
\begin{thebibliography}{10}


\bibitem[Ch11]{Ch11}
G.~Cherlin
\newblock Two problems on homogeneous structures, revisited, 
\newblock {\it in} Proceedings of the AMS-ASL
Special Session on Model Theoretic Methods in Finite Combinatorics, 
January 5-8, 2009, in Washington, DC.;
pp. 319-416 in Contemporary Mathematics, AMS, Providence, RI, 2011.

\bibitem[CS05]{CS05}
G.~Cherlin and S.~Shelah,
\newblock  Universal graphs with a forbidden subtree.
\newblock {\sl J.~Comb.~Theory, Series B} {\bf 97} (2007), 293--333.


\bibitem[CT07]{CT07}
G.~Cherlin and L.~Tallgren, 
\newblock Universal graphs with a forbidden near-path or 2-bouquet,
\newblock {\sl J.~Graph Theory} {\bf 56} (2007), 41--63.

\bibitem[D\v z05]{Dz05}
M.~D\v zamonja, 
\newblock  Club guessing and the universal models.
\newblock {\sl Notre Dame J.~Formal Logic}  {\bf 46}  (2005),   283--300.
 
\bibitem[FK97]{FK97a}
Z.~F\H{u}redi and P.~Komj\'{a}th,
\newblock On the existence of countable universal graphs. 
\newblock {\sl J.~Graph Theory}, {\bf 25} (1997), 53--58.

\bibitem[He71]{He71}
C.~Ward~Henson,
\newblock  A family of countable homogeneous graphs.
\newblock {\sl Pacific J. Math.}  38:69--83, 1971.

\bibitem[KP91]{KP91}
P.~Komj\'{a}th and J.~Pach, 
\newblock 
Universal elements and the complexity of certain classes of infinite graphs.
\newblock {\sl Discrete Math}. {\bf 95} (1991), 255--270.


\end{thebibliography}
